\documentclass[11pt,a4paper]{amsart}

\usepackage{amsmath, amsthm, amssymb}

\newcommand{\bC}{\mathbb{C}}
\newcommand{\bP}{\mathbb{P}}
\newcommand{\bN}{\mathbb{N}}
\newcommand{\bZ}{\mathbb{Z}}
\newcommand{\bD}{\mathbb{D}}
\newcommand{\bR}{\mathbb{R}}
\newcommand{\bQ}{\mathbb{Q}}
\newcommand{\rd}{\mathrm{d}}
\newcommand{\Fix}{\operatorname{Fix}}
\newcommand{\Per}{\operatorname{Per}}
\newcommand{\diam}{\operatorname{diam}}
\newcommand{\Res}{\operatorname{Res}}

\newcommand{\cF}{\mathcal{F}}
  
\numberwithin{equation}{section}
\theoremstyle{plain}
\newtheorem{theorem}{Theorem}[section]
\newtheorem{lemma}{Lemma}[section]
\newtheorem{claim}{Claim}

\newtheorem{mainth}{Theorem}

\theoremstyle{definition}

\newtheorem{notation}{Notation}[section]
\newtheorem*{acknowledgement}{Acknowledgement}
\theoremstyle{remark}
\newtheorem{remark}{Remark}[section]

\begin{document}  
  
\title[Nevanlinna theory in the unicritical polynomials family]{Nevanlinna theory and 
Value distribution in the unicritical polynomials family}

\author[Y\^usuke Okuyama]{Y\^usuke Okuyama}
\address{
Division of Mathematics,
Kyoto Institute of Technology,
Sakyo-ku, Kyoto 606-8585 Japan.}
\email{okuyama@kit.ac.jp}

%\dedicatory{Dedicated to Professor Masahiko Taniguchi 
%on his sixty fifth birthday}

%\thanks{Partially supported by JSPS Grant-in-Aid for Young Scientists (B), 24740087.}

\date{\today}

\begin{abstract}
In the space $\bC$ of the parameters $\lambda$ of
the unicritical polynomials family
$f(\lambda,z)=f_\lambda(z)=z^d+\lambda$ of degree $d>1$,
we establish a quantitative equidistribution result towards the bifurcation current 
(indeed measure) $T_f$ of $f$
as $n\to\infty$ on the averaged distributions of all parameters 
$\lambda$ such that
$f_\lambda$ has a superattracting periodic point of period $n$ in $\bC$,
with a concrete error estimate for $C^2$-test functions on $\bP^1$. 
In the proof, not only complex dynamics but also  
a standard argument from the Nevanlinna theory play key roles.
\end{abstract}


\maketitle

\section{Introduction}\label{sec:intro}

Let $f:\bC\times\bP^1\to\bP^1$ be the {\itshape $($monic and centered$)$
unicritical polynomials family}
\begin{gather}
 f(\lambda,z)=f_{\lambda}(z):=z^d+\lambda
\quad\text{for every }(\lambda,z)\in\bC\times\bP^1\label{eq:unicritical}
\end{gather}
of degree $d>1$. Let $c_0\equiv 0$ on $\bC$, which is
a {\itshape marked critical point} 
of the family $f$ in that for every $\lambda\in\bC$,
$c_0(\lambda)$ is a critical point of $f_\lambda(z)\in\bC[z]$.
For every $n\in\bN\cup\{0\}$, let us define
the monic polynomial
\begin{gather*}
 F_n(\lambda):=f_\lambda^n(c_0(\lambda))\equiv f_\lambda^n(0)\in\bZ[\lambda]
\end{gather*}
of degree $d^{n-1}$. 
Any zero of $F_n$ is simple (Douady--Hubbard \cite[Expos\'e XIX]{DouadyHubbard85}; see also \cite[Theorem 10.3]{MS94} for a simple proof).
The study of the asymptotic behavior as $n\to\infty$ of the set of all zeros of 
$F_n$,
which is the set of all parameters $\lambda\in\bC$ such that $f_\lambda$ has
a superattracting periodic point of (not necessarily exact) period $n$ in $\bC$,
was initiated by Levin \cite{Levin90}, and has been developed by
Bassanelli--Berteloot \cite{BassanelliBerteloot09,BassanelliBerteloot11} and Buff--Gauthier \cite{BG14}
subsequently. 

Our aim is,
from both complex dynamics and the {\itshape Nevanlinna theory},
to contribute to the quantitative study of the asymptotic behavior of 
zeros of $F_n$ 
as $n\to\infty$, partly sharpening Gauthier--Vigny \cite{GV15}. 

\begin{notation}
 Let $\mu:\bN\mapsto\{-1,0,1\}$ be the M\"obius function 
 from arithmetic (cf.\ \cite[\S 2]{Apostol}). Let $\log^+t:=\log\max\{1,t\}$
 on $\bR$.
 Let $\omega$ be the Fubini-Study area element on $\bP^1$ normalized as
 $\omega(\bP^1)=1$, let $[z,w]$ be the chordal metric on $\bP^1$
 normalized as $[\cdot,\infty]=1/\sqrt{1+|\cdot|^2}$ on $\bP^1$
 (following the notation in Nevanlinna's and Tsuji's books \cite{Nevan70,Tsuji59}), 
 and
 let $\delta_x$ be the Dirac measure on $\bP^1$ 
 at each $x\in\bP^1$. The Laplacian $\rd\rd^c$ on $\bP^1$ is normalized as
 $\rd\rd^c(-\log[\cdot,\infty])=\omega-\delta_\infty$ on $\bP^1$.
 Set $\bD(x,r):=\{y\in\bC:|x-y|<r\}$ for every $x\in\bC$ and every $r>0$, 
 $\bD(r):=\bD(0,r)$ for every $r>0$, and $\bD:=\bD(1)$.
\end{notation}

\subsection{Main result}
Let $g_{I_{c_0}}$ be the Green function with pole $\infty$
on the {\itshape escaping} locus
$I_{c_0}:=\{\lambda\in\bC:\limsup_{n\to\infty}|F_n(\lambda)|=\infty\}$ 
of the marked critical point $c_0$ of $f$; $I_{c_0}$ is a punctured open 
and connected neighborhood of $\infty$ in $\bP^1$, and
$\partial I_{c_0}$ and $\bC\setminus I_{c_0}$ respectively coincide 
with the {\itshape $J$-unstability} or {\itshape bifurcation} 
locus $B_f$ and the {\itshape connectedness locus}
$M_f$ of $f$. The function $g_{I_{c_0}}$ 
extends to $\bC$ continuously by setting $g_{I_{c_0}}\equiv 0$ on $M_f$, and
$\mu_{B_f}:=\rd\rd^c g_{I_{c_0}}+\delta_\infty$ on $\bP^1$ coincides with
the harmonic measure on $B_f$ with pole $\infty$. 
The {\itshape measure} $(d-1)d^{-1}\mu_{B_f}$ on $\bP^1$
coincides with the {\itshape bifurcation current} 
(indeed measure) $T_f$ of $f$ on $\bP^1$ 
(see Subsection \ref{sec:DH}).
By a refinement of Przytycki's argument on the recurrence of critical orbits
\cite[Proof of Lemma 2]{Przytycki93} and
Buff's upper estimate of the moduli of the derivatives of polynomials
\cite[the proof of Theorem 3]{Buff03},
we will establish the following $L^1(\omega)$ estimate
\begin{gather}
\int_{\bP^1}\bigl|\log|F_n|-d^{n-1}\cdot g_{I_{c_0}}\bigr|\omega
\le \frac{2\log d}{d-1}n+O(1)
\label{eq:lower}
\end{gather}
as $n\to\infty$, with the {\em concrete} coefficient 
$(2\log d)/(d-1)$ of $n$ in the right hand side; a question on 
the best possibility of this estimate \eqref{eq:lower} seems also interesting.
As seen in the proof of \eqref{eq:lower} (in Section \ref{sec:Selberg}),
this may be regarded as a counterpart of H.\ Selberg's theorem
\cite[p.\ 313]{Selberg44}
from the Nevanlinna theory. 

Our principal result is a deduction from \eqref{eq:lower}
of the following {\itshape quantitative} equidistribution 
of the sequence $(F_n^*\delta_0/d^n)$ of the {\em averaged}
distribution of the superattracting parameters of period $n$
towards $(d-1)^{-1}T_f=d^{-1}\mu_{B_f}$ as $n\to\infty$.

\begin{mainth}\label{th:linear}
 Let $f:\bC\times\bP^1\to\bP^1$ be the unicritical $($monic and centered$)$
 polynomials family of degree $d>1$ defined as in \eqref{eq:unicritical}. Then
for every $\phi\in C^2(\bP^1)$, 
\begin{multline}
 \left|\int_{\bP^1}\phi\rd\left((d-1)\cdot F_n^*\delta_0-d^n\cdot T_f\right)\right|\\
 \le\biggl(\sup_{\bP^1}\biggl|\frac{\rd\rd^c\phi}{\omega}\biggr|\biggr)
 \cdot\bigl((2\log d)n+O(1)\bigr)\label{eq:superattcurrent}
\end{multline}
as $n\to\infty$, where the implicit constant in $O(1)$ is independent of $\phi$
and the Radon-Nikodim derivative $(\rd\rd^c\phi)/\omega$ on $\bP^1$
is bounded on $\bP^1$.
\end{mainth}

For a former application of
Selberg's theorem (Theorem \ref{th:Selberg})
to obtain a quantitative equidistribution result
in complex dynamics, see Drasin and the author \cite{DOproximity}.
As an order estimate, the estimate \eqref{eq:superattcurrent}
is due to Gauthier--Vigny \cite[Theorem A]{GV15}. 
The implicit constant in $O(1)$ in \eqref{eq:superattcurrent}
will also be computed in the proof.
The coefficient $2\log d$
of $n$ in \eqref{eq:superattcurrent} comes from 
the full strength of de Branges's theorem (the solution of the Bieberbach conjecture), on which the proof of Buff's estimate mentioned above essentially
relies. 

\subsection{Non-repelling parameters having exact periods}\label{sec:nonrepelling}

For every $n\in\bN$, the $n$-th {\itshape dynatomic polynomial} 
\begin{gather*}
 \Phi_{f,n}^*(\lambda,z):=\prod_{m\in\bN:\,m|n}(f_\lambda^m(z)-z)^{\mu(n/m)}
\end{gather*}
of the family $f$ is in fact in $\bZ[\lambda,z]$, and 
for every $\lambda\in\bC$, $\Phi_{f,n}^*(\lambda,z)\in\bC[z]$ 
is {\itshape monic} and 
of degree
\begin{gather}
 \nu(n)=\nu_d(n):=\sum_{m\in\bN:\,m|n}\mu\left(\frac{n}{m}\right)d^m.\label{eq:dynatomicdegree}
\end{gather}  
 For every $\lambda\in\bC$ and every $n\in\bN$, let $\Fix_f(\lambda,n)$ be the set
 of all fixed points of $f_{\lambda}^n$ in $\bC$ and set
 $\Fix_f^*(\lambda,n):=\Fix_f(\lambda,n)\setminus
 \bigl(\bigcup_{m\in\bN:\, m|n\text{ and }m<n}\Fix_f(\lambda,m)\bigr)$, 
 each element in which is called
 a periodic point of $f_{\lambda}$ in $\bC$ having the {\itshape exact} period $n$.
 For every $n\in\bN$ and every $\lambda\in\bC$, 
 a periodic point $z$ of $f_{\lambda}$ in $\bC$ is said 
 to have the {\itshape formally exact} period $n$
 if either (i) $z\in\Fix_f^*(\lambda,n)$ or (ii) there is an 
 $m\in\bN$ satisfying $m|n$ and $m<n$ such that
 $z\in\Fix_f^*(\lambda,m)$ and that
 $(f_{\lambda}^m)'(z)$ is a primitive $(n/m)$-th root of unity
 (so in particular $(f_{\lambda}^n)'(z)=1$).
 For every $\lambda\in\bC$ and every $n\in\bN$, let $\Fix_f^{**}(\lambda,n)$
 be the set of all periodic points of $f_\lambda$ in $\bC$ having
 the formally exact period $n$, which in fact coincides with 
 $(\Phi_{f,n}^*(\lambda,\cdot))^{-1}(0)$. 
For every $n\in\bN$, the $n$-th {\itshape multiplier polynomial}
\begin{gather*}
 p_{f,n}^*(\lambda,w)
 :=\Biggl(\prod_{z\in\Fix_f^{**}(\lambda,n)}((f_{\lambda}^n)'(z)-w)\Biggr)^{1/n}
\end{gather*}
of $f$, where for each $\lambda\in\bC$, the product in the right hand side
takes into account the multiplicity of each $z\in\Fix_f^{**}(\lambda,n)$
as a zero of $\Phi_{f,n}^*(\lambda,\cdot)$, 
is indeed in $\bZ[\lambda,w]$
and unique up to multiplication in $n$-th roots of unity. 
For every $w\in\bC$, by a direct computation, 
\begin{gather}
\deg_\lambda p_{f,n}^*(\lambda,w)
=(d-1)\frac{\nu(n)}{d}\label{eq:factorizationdegree}
\end{gather}
and the coefficient of the leading term of $p_{f,n}^*(\lambda,w)\in\bC[\lambda]$
equals $d^{\nu(n)}$,
both of which are independent of $w$. For every $n\in\bN$ and every $w\in\bC$, let
$\Per_f^*(n,w)$ be the effective divisor on $\bP^1$ defined by the zeros
of $p_{f,n}^*(\lambda,w)\in\bC[\lambda]$; as a Radon measure on $\bP^1$,
\begin{gather*}
 \Per_f^*(n,w)=\rd\rd^c_\lambda\log|p_{f,n}^*(\lambda,w)|+(d-1)\frac{\nu(n)}{d}\delta_\infty.
\end{gather*} 
 For more details,
 see e.g.\ \cite[\S 4]{SilvermanDynamics}, \cite[\S 2.3]{BertelootCIME},
 \cite[\S 3]{MortonVivaldi95}.

\begin{notation}\label{th:sigma}
 Let $(\sigma_0(n))$ and $(\sigma_1(n))$
 be such sequences in $\bN$ that 
 $1=\sum_{m\in\bN:\,m|n}\mu(n/m)\sigma_0(m)$ and
 $n=\sum_{m\in\bN:\,m|n}\mu(n/m)\sigma_1(m)$,
 or equivalently, $\sigma_0(n)=\sum_{m\in\bN:\,m|n}1$
 and $\sigma_1(n)=\sum_{m\in\bN:\,m|n}m$ by M\"obius inversion,
 for every $n\in\bN$.
\end{notation}

By an argument similar to that in the proof of Theorem \ref{th:linear},
we will also show the following.

\begin{mainth}\label{th:nonrepelling}
Let $f:\bC\times\bP^1\to\bP^1$ be the unicritical $($monic and centered$)$
polynomials family of degree $d>1$ defined as in \eqref{eq:unicritical}. Then
for every $\phi\in C^2(\bP^1)$, 
\begin{multline}
 \left|\int_{\bP^1}\phi\rd\left(\Per_f^*(n,0)-\nu(n)\cdot T_f\right)\right|\\
 \le\biggl(\sup_{\bP^1}\biggl|\frac{\rd\rd^c\phi}{\omega}\biggr|\biggr)
 \cdot \bigl((2\log d)\sigma_1(n)+O(\sigma_0(n))\bigr)\label{eq:superattprimitive}
\end{multline}
as $n\to\infty$, 
where the term $O(\sigma_0(n))$ is independent of $\phi$, and
for every $\phi\in C^2(\bP^1)$ and every $r\in(0,1]$, 
\begin{multline}
  \left|\int_{\bP^1}\phi\rd\left(\int_0^{2\pi}\Per_f^*(n,re^{i\theta})\frac{\rd\theta}{2\pi}-\nu(n)\cdot T_f\right)\right|\\
 \le\biggl(\sup_{\bP^1}\biggl|\frac{\rd\rd^c\phi}{\omega}\biggr|\biggr)
 \cdot \bigl((2\log d)\sigma_1(n)+O(\sigma_0(n))\bigr)\label{eq:nonrepaveraged}
\end{multline}
as $n\to\infty$, where the term $O(\sigma_0(n))$ 
is independent of both $\phi$ and $r$. 
Here the Radon-Nikodim derivative $(\rd\rd^c\phi)/\omega$ on $\bP^1$
is bounded on $\bP^1$.
\end{mainth}

Again, the terms $O(\sigma_0(n))$ in Theorem \ref{th:nonrepelling}
will also be computed
%, as precise as possible, 
in Section \ref{sec:nonrepellingproof}.
As an order estimate, the estimate \eqref{eq:superattprimitive}
is a consequence of Gauthier--Vigny \cite[Theorem A]{GV15}. The estimate \eqref{eq:nonrepaveraged}
quantifies Bassanelli--Berteloot \cite[2.\ in Theorem 3.1]{BassanelliBerteloot11} 
for $r\in(0,1]$.

\subsection{Organization of the article}
In Section \ref{sec:background}, we recall background from the study
of the unicritical polynomials family $f$. 
In Section \ref{sec:Selberg}, we show 
Theorem \ref{th:linear}. 
In Section \ref{sec:nonrepellingproof}, 
we show Theorem \ref{th:nonrepelling}.

\section{Background from the study of the family $f$}
\label{sec:background}

Let $f:\bC\times\bP^1\to\bP^1$ be the unicritical $($monic and centered$)$
polynomials family of degree $d>1$ defined as in \eqref{eq:unicritical}, and recall that 
$c_0(\lambda)=0\in\bZ[\lambda]$ defines a marked critical point of $f$.

\subsection{
Douady--Hubbard's theory on the parameter space $\bC$ of $f$}\label{sec:DH}

For every $\lambda\in\bC$, 
let $J_{f_\lambda}$ be the Julia set of $f_\lambda$, which is
compact in $\bC$.
Let $B_f$ be the {\itshape $J$-unstability} or {\itshape bifurcation} 
locus of the family $f$, which is the discontinuity locus
of the set function $\lambda\mapsto J_{f_\lambda}$ 
with respect to the Hausdorff topology from $(\bP^1,[z,w])$, and is
closed and nowhere dense in $\bC$
(by Ma\~n\'e--Sad--Sullivan \cite{MSS}, Lyubich \cite{Lyubich83stability}).
The {\itshape escaping} locus 
\begin{gather*}
 I_{c_0}:=\{\lambda\in\bC:\limsup_{n\to\infty}|F_n(\lambda)|=\infty\} 
\end{gather*}
of the marked critical point $c_0$ of $f$
is a punctured open and connected neighborhood of $\infty$ in $\bP^1$ and
coincides with the unique unbounded component of $\bC\setminus B_f$.
We have $B_f=\partial I_{c_0}$, and
the {\itshape connectedness} locus 
\begin{gather*}
 M_f:=\{\lambda\in\bC:J_{f_\lambda}\text{ is connected}\}
\end{gather*}
of $f$ coincides with $\bC\setminus I_{c_0}$ (and is connected).
For every $\lambda\in\bC$,
$f_\lambda$ has at most one non-repelling cycle in $\bC$
(see, e.g., \cite[\S 8]{Milnor3rd}).
Let $H_f$ be the {\itshape hyperbolicity} locus of $f$,
which coincides with the union of $I_{c_0}$ and
the set of all $\lambda\in M_f$ such that
$f_\lambda$ has the (super)attracting cycle in $\bC$, and
is a closed and open subset in $\bC\setminus B_f$. 
For example, for every $n\in\bN$, $0\in F_n^{-1}(0)\subset H_f\setminus I_{c_0}$.
For every component $U$ of $H_f\setminus I_{c_0}$,
there are an $n_U\in\bN$ and a proper holomorphic mapping
$\phi_U:U\to \bD$ of degree $d-1$ such that $\#\phi_U^{-1}(0)=1$
and that for every $w\in\bD$, 
$\phi_U^{-1}(w)$ coincides with the set of all $\lambda\in U$
such that $f_\lambda$ has the (super)attracting cycle in $\bC$
having the {\itshape exact} period $n_U$ and the multiplier $w$.
For more details, see Douady--Hubbard \cite{DH82}, and for
a modern treatment, see McMullen--Sullivan \cite[\S 10]{MS94}.

\subsection{The Green functions on the dynamical and parameter spaces}\label{sec:Green}
For every $\lambda\in\bC$, $J_{f_\lambda}$
coincides with the boundary of the filled-in Julia set 
$K_{f_\lambda}:=\{z\in\bC:\limsup_{n\to\infty}|f_\lambda^n(z)|<\infty\}$ 
of $f_\lambda$, which is compact in $\bC$. 
For every $\lambda\in\bC$, the uniform limit
\begin{gather}
g_{f_\lambda}(z):=\lim_{n\to\infty}\frac{-\log[f_\lambda^n(z),\infty]}{d^n}\label{eq:Greendynamical}
\end{gather}
exists on $\bC$, 
%the restriction of $g_{f_\lambda}$ to 
%$\bC\setminus K_{f_\lambda}$ coincides with the Green function on 
%$\bC\setminus K_{f_\lambda}$ with pole $\infty$, 
and setting $g_{f_\lambda}(\infty):=+\infty$, the probability measure
$\mu_{f_\lambda}:=\rd\rd^c g_{f_\lambda}+\delta_\infty$ 
on $\bP^1$ coincides with the harmonic measure on $J_{f_\lambda}$ with pole
$\infty$. Moreover, $\mu_{f_\lambda}$ is mixing so ergodic under $f_\lambda$ (by Brolin \cite{Brolin}).
For completeness, we include a proof of the following.

\begin{lemma}\label{th:locallybounded}
For every $\lambda\in\bC$, 
\begin{gather}
\sup_{\bC}\bigl|g_{f_\lambda}+\log[\cdot,\infty]\bigr|
\le\frac{1}{d-1}\cdot\sup_{z\in\bC}
\biggl|\log\frac{[z,\infty]^d}{[f_\lambda(z),\infty]}\biggr|,\label{eq:Green}
\end{gather}
and the function
$\lambda\mapsto \sup_{z\in \bC}|\log([z,\infty]^d/[f_\lambda(z),\infty])|$
is locally bounded on $\bC$.
\end{lemma}

\begin{proof}
For every $\lambda\in\bC$, by the definition \eqref{eq:Greendynamical} of $g_{f_\lambda}$,
we have
\begin{multline*}
\sup_{\bC}\bigl|g_{f_\lambda}+\log[\cdot,\infty]\bigr|
\le\sup_{z\in\bC}\Biggl|\sum_{j=1}^\infty
\frac{-\log[f_\lambda(f_\lambda^{j-1}(z)),\infty]+d\cdot\log[f_\lambda^{j-1}(z),\infty]}{d^j}\Biggr|\\
\le\frac{1}{d-1}\cdot\sup_{z\in\bC}
\biggl|\log\frac{[z,\infty]^d}{[f_\lambda(z),\infty]}\biggr|.
\end{multline*}
For every $\lambda\in\bC$, let us define 
the non-degenerate homogeneous polynomial endomorphism
$\tilde{f}_\lambda:\bC^2\to\bC^2$ of degree $d$ by
$\tilde{f}_\lambda(p_0,p_1):=(p_0^d,p_0^df_\lambda(p_1/p_0))
=(p_0^d,p_1^d+\lambda p_0^d)$. Then 
the function $(\lambda,(p_0,p_1))\mapsto
\bigl|\log\|\tilde{f}_\lambda(p_0,p_1)\|\bigr|$ is continuous 
on $\bC\times(\bC^2\setminus\{(0,0)\})$, and for every compact subset
$K$ in $\bC$, we have
\begin{gather*}
\sup_{(\lambda,z)\in K\times\bC}\biggl|\log\frac{[z,\infty]^d}{[f_\lambda(z),\infty]}\biggr|=
\sup_{(\lambda,(p_0,p_1))\in K\times S(1)}\bigl|\log\|\tilde{f}_\lambda(p_0,p_1)\|\bigr|,
\end{gather*}
where $\|\cdot\|$ is the Euclidean norm on $\bC^2$ and
$S(1):=\{(p_0,p_1)\in\bC^2:\|(p_0,p_1)\|=1\}$.
Now the proof is complete by the compactness of $K$ in $\bC$ and that of 
$S(1)$ in $\bC^2\setminus\{(0,0)\}$.
\end{proof}

Similarly, the locally uniform limit
\begin{gather*}
\lambda\mapsto g_{I_{c_0}}(\lambda):=\lim_{n\to\infty}\frac{-\log[F_n(\lambda),\infty]}{d^{n-1}}
=d\cdot g_{f_\lambda}(c_0(\lambda))=g_{f_\lambda}(f_{\lambda}(c_0(\lambda)))
\end{gather*}
exists on $\bC$,
%The restriction of $g_{I_{c_0}}$ to $I_{c_0}$
%coincides with the Green function on $I_{c_0}$ with pole $\infty$, 
and setting $g_{I_{c_0}}:=+\infty$, the probability measure
\begin{gather*}
 \mu_f:=\rd\rd^c g_{I_{c_0}}+\delta_\infty\quad\text{on }\bP^1
\end{gather*}
coincides with the harmonic measure on $B_f=\partial I_{c_0}$ with pole $\infty$ 
(by Douady--Hubbard \cite{DH82}, Sibony \cite{SibonyUCLA}). 
The {\itshape activity current} (indeed measure)
of the marked critical point $c_0$ of $f$ is
\begin{gather*}
 T_{c_0}:=\lim_{n\to\infty}\frac{F_n^*\omega}{d^n}=\frac{\mu_f}{d}
\end{gather*}
as currents on $\bP^1$
(DeMarco \cite{DeMarco01}, Dujardin--Favre \cite{DujardinFavre08}).
For every $\lambda\in\bC$, the {\itshape Lyapunov exponent} of $f_\lambda$ 
with respect to $\mu_{f_\lambda}$ is
\begin{gather*}
 L(f_\lambda):=\int_{\bP^1}\log|f_\lambda'(z)|\rd\mu_{f_\lambda}(z)
 =\log d+(d-1)\frac{g_{I_{c_0}}}{d}(\ge\log d>0)
\end{gather*}
(Manning \cite{Manning84}, Przytycki \cite{Przytycki85}).  
Setting $L(f_\lambda)|_{\lambda=\infty}:=+\infty$,
the {\itshape bifurcation current} of $f$ can be defined by
\begin{gather}
 T_f:=\rd\rd^c L(f_\cdot)+\frac{d-1}{d}\delta_\infty
=(d-1)\frac{\mu_f}{d}=(d-1)T_{c_0}
\quad\text{on }\bP^1\label{eq:bifcurrent}
\end{gather}
(DeMarco \cite{DeMarco03}). 
For more details, see, e.g., Berteloot's survey \cite[\S 3.2.3]{BertelootCIME}. 

\section{Proof of Theorem \ref{th:linear}}
\label{sec:Selberg}

Let $f:\bC\times\bP^1\to\bP^1$ be the unicritical polynomials family
of degree $d>1$ defined as \eqref{eq:unicritical}.
For every $\lambda\in\bC$ and every $n\in\bN$, 
let us define the {\itshape chordal derivative}
\begin{gather*}
 (f_\lambda^n)^\#:=\sqrt{\frac{(f_\lambda^n)^*\omega}{\omega}}:\bP^1\to\bR_{\ge 0}
\end{gather*}
of $f_\lambda^n$ on $\bP^1$. For every non-empty subset $S$ in $\bP^1$, let
$\diam_\#(S)$ be the chordal diameter of $S$.
The resultant of $(P(z),Q(z))\in\bC[z]\times\bC[z]$ is denoted by $\Res(P,Q)$, as usual. Recall that $\{z\in\bC:[z,0]<[r,0]\}=\bD(0,r)$ for every $r>0$ and
that $[z,w]\le|z-w|$ on $\bC\times\bC$.

\begin{lemma}\label{th:lower}
For every $n\in\bN$ and every $\lambda\in\bC\setminus(H_f\setminus I_{c_0})$
$($so in particular for every $\lambda\in B_f)$,
\begin{gather*}
|F_n(\lambda)|
\ge\bigl(\sqrt{2}-1\bigr)\Bigl(2^{d+1}\cdot\sup_{z\in\bP^1}((f_\lambda^{n-1})^\#(z))\Bigr)^{-1/(d-1)}.
\end{gather*}
\end{lemma}

\begin{proof}
Fix $n\in\bN$, and define the functions
$L_{n-1}$ and $\epsilon_n$ on $\bC$ by
$L_{n-1}(\lambda):=\sup_{z\in \bP^1}((f_\lambda^{n-1})^\#(z))(>1)$
and $\epsilon_n(\lambda):=(2^2\cdot L_{n-1}(\lambda))^{-1/(d-1)}(<1)$. 
For every $\lambda\in\bC$, noting that $f_\lambda(0)=\lambda$
and that $f_\lambda(z)-f_\lambda(0)=z^d$ on $\bC$, we have
\begin{multline*}
 \diam_\#\bigl(f_\lambda^n(\{z\in\bC:[z,0]<[\epsilon_n(\lambda),0]\})\bigr)
=\diam_\#\bigl(f_\lambda^n(\bD(0,\epsilon_n(\lambda)))\bigr)\\
=\diam_\#\bigl(f_\lambda^{n-1}(\bD(\lambda,\epsilon_n(\lambda)^d))\bigr)\\
 \le L_{n-1}(\lambda)\cdot\diam_\#(\bD(\lambda,\epsilon_n(\lambda)^d))
 \le L_{n-1}(\lambda)\cdot 2\epsilon_n(\lambda)^d=\frac{\epsilon_n(\lambda)}{2}, 
\end{multline*}
so that if $[f_\lambda^n(0),0]<[\epsilon_n(\lambda),0]-\epsilon_n(\lambda)/2$, then
$\sup\bigl\{[w,0]:w\in f_\lambda^n(\{z\in\bC:[z,0]\le[\epsilon_n(\lambda),0]\})\bigr\}
<([\epsilon_n(\lambda),0]-\epsilon_n(\lambda)/2)+\epsilon_n(\lambda)/2
=[\epsilon_n(\lambda),0]$, i.e.,
$f_\lambda^n(\{z\in\bC:[z,0]<[\epsilon_n(\lambda),0]\})\Subset\{z\in\bC:[z,0]<[\epsilon_n(\lambda),0]\}$; then by Brouwer's fixed point theorem, 
Montel's theorem, and Fatou's classification of cyclic Fatou components 
(see e.g.\ \cite[\S 16]{Milnor3rd}), 
the domain $\{z\in\bC:[z,0]<[\epsilon_n(\lambda),0]\}$,
which contains both the critical point $c_0(\lambda)(=0)$ of $f_\lambda$
and a fixed point of $f_\lambda^n$, 
is contained in the immediate basin of
a (super)attracting cycle of $f_\lambda$ in $\bC$.

Hence for every $\lambda\in\bC$, we obtain the desired lower estimate
\begin{multline*}
 |F_n(\lambda)|\ge([F_n(\lambda),0]=)[f_\lambda^n(0),0]
\ge[\epsilon_n(\lambda),0]-\frac{\epsilon_n(\lambda)}{2}\\
\ge\bigl(\sqrt{2}-1\bigr)\frac{\epsilon_n(\lambda)}{2}
=\bigl(\sqrt{2}-1\bigr)(2^{d+1}L_{n-1}(\lambda))^{-1/(d-1)}
\end{multline*}
of $|F_n(\lambda)|$ unless $0$ is in the immediate basin of
a (super)attracting cycle of $f_\lambda$ in $\bC$.
%
%If $f_{\lambda_0}$ has a parabolic cycle (in $\bC$) for some $\lambda_0\in\bC$,
%then there are $n_0\in\bN$ and a root $w_0\in\partial\bD$ of unity such that 
%$\Res\bigl(f_\lambda^{n_0}(\cdot)-\Id,(f_\lambda^{n_0})'(\cdot)=w\bigr)\in\bZ[\lambda,w]$
%vanishes at $(\lambda_0,w_0)$. Then
%by Rothe's theorem, there are $(w_j)_{j=1}^\infty$ 
%and $(\lambda_j)_{j=1}^\infty$ in $\bC$ tending to $w_0$ and $\lambda_0$
%as $j\to\infty$, respectively, such that
%for every $j\in\bN$, $w_j$ is not a root of unity but in $\partial\bD$,
%and $\Res\bigl(f_\lambda^{n_0}(\cdot)-\Id,(f_\lambda^{n_0})'(\cdot)=w\bigr)\in\bZ[\lambda,w]$
%vanishes at $(\lambda_j,w_j)$. In particular,
%for every $j\in\bN$, 
%$f_{\lambda_j}$ has neither (super)attracting nor parabolic cycles in $\bC$.
%
%Now the continuity of $F_n$
%and the upper semicontinuity of $L_{n-1}$ on $\bC$ 
Now the proof is complete.
\end{proof}

The following is substantially shown in Buff \cite[the proof of Theorem 4]{Buff03}.
\begin{theorem}[Buff]\label{th:deBranges}
Let $f\in\bC[z]$ be of degree $d>1$, and let $z_0\in\bC$.
If $g_f(z_0)\ge\max_{c\in C(f)\cap\bC}g_f(c)$, where 
$g_f$ is the Green function of the filled-in Julia set $K_f$ of $f$
with pole $\infty$ and
$C(f)$ is the set of all critical points of $f$,
then $|f'(z_0)|\le d^2\cdot e^{(d-1)g_f(z_0)}$, and
the equality {\itshape never} holds if 
$C(f)\cap\bC$ is not contained in $K_f$.
\end{theorem}

\begin{lemma}\label{th:Buffchordal}
For every $n\in\bN$ and every $\lambda\in M_f$,
\begin{gather*}
\log\biggl(\sup_{z\in\bP^1}((f_\lambda^n)^\#(z))\biggr)
\le (2\log d)n
+\frac{4}{d-1}\cdot\sup_{z\in\bC}
\biggl|\log\frac{[z,\infty]^d}{[f_\lambda(z),\infty]}\biggr|.
\end{gather*}
\end{lemma}

\begin{proof}
 For every $n\in\bN$, every $\lambda\in M_f$, and every $z\in\bC$,
 by Theorem \ref{th:deBranges}, we have
 $|(f_\lambda^n)'(z)|\le 
 (d^n)^2e^{(d^n-1)g_{f_\lambda}(z)}$, and
 by the definition \eqref{eq:Greendynamical} of $g_{f_\lambda}$, we have
 $0\le(d^n-1)g_{f_\lambda}(z)=g_{f_\lambda}(f_\lambda^n(z))-g_{f_\lambda}(z)$,
 so that
 \begin{align*}
 (f_\lambda^n)^\#(z)
  =&|(f_\lambda^n)'(z)|\cdot
 \frac{[f_\lambda^n(z),\infty]^2}{[z,\infty]^2}\\
 \le& 
 d^{2n}e^{g_{f_\lambda}(f_\lambda^n(z))-g_{f_\lambda}(z)}
 \cdot e^{2(\log[f_\lambda^n(z),\infty]-\log[z,\infty])}\\
\le& d^{2n}\cdot e^{2(g_{f_\lambda}(f_\lambda^n(z))+\log[f_\lambda^n(z),\infty])
 -2(g_{f_\lambda}(z)+\log[z,\infty])
}\\
\le & d^{2n}\cdot e^{4\sup_{\bC}|g_{f_\lambda}+\log[\cdot,\infty]|}.
\end{align*}
This with \eqref{eq:Green} completes the proof.
\end{proof}

Recalling the latter half of Lemma \ref{th:locallybounded}, we can set
\begin{gather*}
 C_{B_f}:=\sup_{(\lambda,z)\in B_f\times\bC}\left|\log\frac{[z,\infty]^d}{[f_\lambda(z),\infty]}\right|<\infty.
\end{gather*}
Then for every $n\in\bN$, by Lemmas \ref{th:lower} and \ref{th:Buffchordal}, we have 
\begin{multline*}
\inf_{B_f}\log|F_n|\\
\ge-\frac{1}{d-1}\Bigl((d+1)\log 2+(2\log d)(n-1)+\frac{4C_{B_f}}{d-1}\Bigr)+
\log\bigl(\sqrt{2}-1\bigr).
\end{multline*}
On the other hand, for every $n\in\bN$ and every $\lambda\in M_f$, 
by Buff \cite[Theorem 1]{Buff03}, we also have
$F_n(\lambda)=f_\lambda^n(c_0(\lambda))\in K_{f_\lambda}\subset\bD(2)$.
Hence for every $n\in\bN$, we have the following uniform estimate
\begin{multline}
\sup_{B_f}\bigl|\log|F_n|\bigr|\\
\le\frac{1}{d-1}\Bigl((d+1)\log 2+(2\log d)(n-1)+\frac{4C_{B_f}}{d-1}
+(d-1)\log\bigl(\sqrt{2}+1\bigr)\Bigr)=:t_n.
\label{eq:lowerconcrete}
\end{multline}

Now let us recall the following classical theorem from the Nevanlinna theory;
for a modern formulation, see \cite{Weitsman72}.

\begin{theorem}[Selberg {\cite[p.\ 311]{Selberg44}}]\label{th:Selberg}
Let $V$ be a bounded and at most finitely connected domain in $\bC$ whose
boundary components are piecewise real analytic Jordan closed curves, 
so that for every $y\in V$, 
the Green function $G_V(\cdot,y)$ on $V$ with pole $y$ exists and
extends continuously to $\bC$ by setting $\equiv 0$ on $\bC\setminus V$.
If $V$ is in $\bC\setminus\{0\}$, 
then for every $y\in V$ and every $r>0$, setting 
$\theta_V(r):=\int_{\{\theta\in[0,2\pi]:\,re^{i\theta}\in V\}}\rd\theta\in[0,2\pi]$, 
we have
\begin{gather}
 \int_0^{2\pi}G_V(re^{i\theta},y)\frac{\rd\theta}{2\pi}
 \le\min\biggl\{\frac{\pi}{2}\tan\frac{\theta_V(r)}{4},
 \log^+\frac{r}{\inf_{z\in V}|z|}\biggr\}.\label{eq:Selberg}
\end{gather} 
\end{theorem}

Let $H_1$ be the component of $H_f$ containing $0$ 
and set
\begin{multline*}
 C_0:=\pi
+\int_0^\infty\frac{2r}{(1+r^2)^2}
\log^+\frac{r}{\sup\{t>0:\bD(t)\subset H_1\}}\rd r\\
+\int_{H_1}G_{H_1}(\cdot,0)\omega<\infty.
\end{multline*}
Fix $n\in\bN$. Recall that $\deg F_n=d^{n-1}$.

\begin{claim}\label{th:small}
\begin{gather*}
  \int_{F_n^{-1}(\bD(e^{-t_n}))}\bigl|\log|F_n|-d^{n-1}\cdot g_{I_{c_0}}
 \bigr|\omega
 \le\omega(F_n^{-1}(\bD(e^{-t_n})))t_n+C_0.
\end{gather*}
\end{claim}

\begin{proof}
By \eqref{eq:lowerconcrete}, we have $\inf_{B_f}|F_n|\ge e^{-t_n}$.
Let $\mathcal{F}$ be 
the family of all components of $F_n^{-1}(\bD(e^{-t_n}))$,
so that $\#\mathcal{F}\le d^{n-1}$. 
By the description of $H_f$ in Subsection \ref{sec:DH},
every $V\in\mathcal{F}$ is a piecewise real analytic
Jordan domain in $H_f\setminus I_{c_0}$ and,
since any zero of $F_n$ is also simple,
for every $V\in\mathcal{F}$, the restriction $F_n|V:V\to\bD(e^{-t_n})$ is 
conformal. For every $V\in\mathcal{F}$, set $\lambda_V:=(F_n|V)^{-1}(0)$.
Let $V_0$ be the element of $\mathcal{F}$ containing $0$.
Recall the notation in Theorem \ref{th:Selberg}. 
For every $V\in\mathcal{F}$,
by the conformal invariance of the Green functions, we have
\begin{gather*}
 \log\frac{e^{-t_n}}{|F_n|}=G_{\bD(e^{-t_n})}(F_n,0)=G_V(\cdot,\lambda_V)\quad\text{on }V.
\end{gather*}
For every $r>0$, fixing such $V_r\in\cF\setminus\{V_0\}$ 
that for every $V\in\mathcal{F}\setminus\{V_0\}$,
$\theta_{V_r}(r)\ge\theta_V(r)$ (so in particular that
for every $V\in\mathcal{F}\setminus\{V_0,V_r\}$, $\theta_V(r)\in[0,\pi]$ 
since $2\pi\ge\theta_{V_r}(r)+\theta_V(r)\ge 2\theta_V(r)\ge 0$),
we have
\begin{align*}
&\sum_{V\in\mathcal{F}}
\int_0^{2\pi}G_V(re^{i\theta},\lambda_V)\frac{\rd\theta}{2\pi}\\
=&\sum_{V\in\mathcal{F}\setminus\{V_0\}}
\int_0^{2\pi}G_V(re^{i\theta},\lambda_V)\frac{\rd\theta}{2\pi}
+\int_0^{2\pi}G_{V_0}(re^{i\theta},0)\frac{\rd\theta}{2\pi}\\
\le&\biggl(\sum_{V\in\mathcal{F}\setminus\{V_0,V_r\}}\Bigl(\frac{\pi}{2}\tan\frac{\theta_V(r)}{4}\Bigr)
+\log^+\frac{r}{\inf_{z\in V_r}|z|}\biggr)
+\int_0^{2\pi}G_{H_1}(re^{i\theta},0)\frac{\rd\theta}{2\pi}\\
\le&\frac{\pi}{2}\cdot\sum_{V\in\mathcal{F}\setminus\{V_0,V_r\}}\frac{\theta_V(r)}{\pi}+\log^+\frac{r}{\sup\{t>0:\bD(t)\subset H_1\}}+\int_0^{2\pi}G_{H_1}(re^{i\theta},0)\frac{\rd\theta}{2\pi}\\
\le&\frac{\pi}{2}\cdot\frac{2\pi}{\pi}+\log^+\frac{r}{\sup\{t>0:\bD(t)\subset H_1\}}
+\int_0^{2\pi}G_{H_1}(re^{i\theta},0)\frac{\rd\theta}{2\pi},
\end{align*}
where the first inequality is 
by \eqref{eq:Selberg} and the monotonicity of the Green functions, and
the second inequality is by $\theta_V(r)\in[0,\pi]$ for every $V\in\cF\setminus\{V_0,V_r\}$.
Hence, since $t_n\ge 0$, we have
\begin{multline*}
\int_{F_n^{-1}(\bD(e^{-t_n}))}\bigl|\log|F_n|\bigr|\omega
=\int_{F_n^{-1}(\bD(e^{-t_n}))}(-\log|F_n|)\omega\\
=\omega(F_n^{-1}(\bD(e^{-t_n})))t_n+
\int_0^\infty\frac{2r\rd r}{(1+r^2)^2}
\sum_{V\in\mathcal{F}}\int_0^{2\pi}G_V(re^{i\theta},\lambda_V)\frac{\rd\theta}{2\pi}\\
\le\omega(F_n^{-1}(\bD(e^{-t_n})))t_n+C_0,
\end{multline*}
which completes the proof.
\end{proof}

\begin{claim}\label{th:harmonic}
$\sup_{\bC\setminus F_n^{-1}(\bD(e^{-t_n}))}\bigl|\log|F_n|-d^{n-1}\cdot g_{I_{c_0}}\bigr|
\le t_n$. 
\end{claim}

\begin{proof}
By the description of $H_f$ in Subsection \ref{sec:DH},
the function $\log|F_n|-d^{n-1}\cdot g_{I_{c_0}}$
is not only harmonic on $I_{c_0}$ but also bounded around $\infty$
so, by the removable singularity theorem for subharmonic functions twice,
extends {\itshape harmonically} to $I_{c_0}\cup\{\infty\}$.
Applying the maximum principle to 
this harmonic extension on $I_{c_0}\cup\{\infty\}$ twice,
by $g_{I_{c_0}}\equiv 0$ on $M_f$ and \eqref{eq:lowerconcrete}, we have
$\sup_{I_{c_0}}\bigl|\log|F_n|-d^{n-1}\cdot g_{I_{c_0}}\bigr|
\le\sup_{B_f}\bigl|\log|F_n|\bigr|\le t_n$ 
(cf.\ \cite[the proof of Lemma 4.1]{GV15}). 
Similarly, 
applying the maximum principle twice to the restriction of $\log|F_n|$
on $M_f\setminus F_n^{-1}(\bD(e^{-t_n}))$, which is harmonic on the
interior of $M_f\setminus F_n^{-1}(\bD(e^{-t_n}))$,
by $g_{I_{c_0}}\equiv 0$ on $M_f$ and \eqref{eq:lowerconcrete},
we have $\sup_{M_f\setminus F_n^{-1}(\bD(e^{-t_n}))}
\bigl|\log|F_n|-d^{n-1}\cdot g_{I_{c_0}}\bigr|
\le\sup_{B_f\cup F_n^{-1}(\partial\bD(e^{-t_n}))}\bigl|\log|F_n|\bigr|\le t_n$.
Now the proof is complete.
\end{proof}

\begin{remark}
 The proof of Claim \ref{th:harmonic} is independent of
the possibility of the existence of a queer component
of the interior of $M_f$.
\end{remark}

By Claims \ref{th:small} and \ref{th:harmonic}, 
we have the following $L^1(\omega)$ estimate
\begin{multline}
\int_{\bP^1}\bigl|\log|F_n|-d^{n-1}\cdot g_{I_{c_0}}\bigr|\omega\\
\le\bigl(\omega(F_n^{-1}(\bD(e^{-t_n})))t_n+C_0\bigr)
+\omega(\bC\setminus F_n^{-1}(\bD(e^{-t_n})))t_n
= t_n+C_0,
\label{eq:superattconcrete}
\end{multline}
so \eqref{eq:lower} holds.

Recalling \eqref{eq:bifcurrent}, we also have
$(d-1)F_n^*\delta_0-d^n\cdot T_f
=(d-1)\cdot\rd\rd^c(\log|F_n|-d^{n-1}\cdot g_{I_{c_0}})$
on $\bP^1$,
so that by Green's theorem, for every $\phi\in C^2(\bP^1)$, the estimate \eqref{eq:superattconcrete} yields
\begin{gather}
 \left|\int_{\bP^1}\phi\rd\left((d-1)\cdot F_n^*\delta_0-d^n\cdot T_f\right)\right|
 \le\biggl(\sup_{\bP^1}\biggl|\frac{\rd\rd^c\phi}{\omega}\biggr|\biggr)
 \cdot(d-1)(t_n+C_0),
\tag{\ref{eq:superattcurrent}$'$}\label{eq:superattproximity}
\end{gather}
so \eqref{eq:superattcurrent} holds. Now the proof of Theorem \ref{th:linear} is complete. \qed

\section{Proof of Theorem \ref{th:nonrepelling}}
\label{sec:nonrepellingproof}
Let $f:\bC\times\bP^1\to\bP^1$ be the unicritical polynomials family
of degree $d>1$ defined as \eqref{eq:unicritical}. Recall the definitions 
(and properties) of
$\Phi_{f,n}^*(\lambda,z)\in\bZ[\lambda,z]$,
$p_{f,n}^*(\lambda,w)\in\bZ[\lambda,z]$,
and $\Fix_f^{**}(\lambda,n)$
in Subsection \ref{sec:nonrepelling}.
For every $n\in\bN$, it would be convenient to set
\begin{gather*}
 P_n^*(\lambda,w)=P_{f,n}^*(\lambda,w):=\frac{p_{f,n}^*(\lambda,w)}{d^{\nu(n)}}
\in\bQ[\lambda,w],
\end{gather*}
so that for every $w\in\bC$, $P_n^*(\lambda,w)\in\bC[\lambda]$ is {\itshape monic}.

\begin{lemma}\label{th:preparatory}
For every $n\in\bN$ and every $\lambda\in\bC$, we have
\begin{multline}
 P_n^*(\lambda,0)
=\bigl((-1)^{\nu(n)}\cdot\Phi_{f,n}^*(\lambda,0)\bigr)^{d-1}\\
=\biggl((-1)^{\nu(n)}\cdot\prod_{m\in\bN:\, m|n}F_m(\lambda)^{\mu(n/m)}\biggr)^{d-1}\label{eq:factorization}
\end{multline}
$($up to multiplication in $n$-th roots of unity$)$.
For every $n>1$, we have $0\not\in(P_n^*(\cdot,0))^{-1}(0)$.
For every $n\in\bN$ and every $\lambda\in\bC$, 
if $\lambda\in(P_n^*(\cdot,0))^{-1}(0)$, then
$(c_0(\lambda)=)0\in\Fix_f^*(\lambda,n)$ 
and $\lambda$ is a zero of $P_n^*(\cdot,0)$ of the order $d-1$.
\end{lemma}

\begin{proof}
For every $n\in\bN$ and every $\lambda\in\bC$, 
by the chain rule and the equalities $f_{\lambda}'(z)=d\cdot z^{d-1}$ 
and $\Fix_f^{**}(\lambda,n)=(\Phi_{f,n}^*(\lambda,\cdot))^{-1}(0)$,
we have
\begin{align*}
 (p_{f,n}^*(\lambda,0))^n
 \biggl(=&\prod_{z\in\Fix_f^{**}(\lambda,n)}(f_{\lambda}^n)'(z)\biggr)
 =d^{\nu(n)n}\bigl((-1)^{\nu(n)}\cdot\Phi_{f,n}^*(\lambda,0)\bigr)^{n(d-1)}\\
 =&d^{\nu(n)n}\biggl((-1)^{\nu(n)}\cdot\prod_{m\in\bN:\,m|n}(f_{\lambda}^m(0)-0)^{\mu(n/m)}\biggr)^{n(d-1)},
\end{align*}
which (with the definition of $F_m$) yields \eqref{eq:factorization}. 
For every $m\in\bN$, even by a direct computation, $0$ is 
a simple zero of $F_m$ in $\bC$, so that for every $n>1$, 
$0\not\in(P_n^*(\cdot,0))^{-1}(0)$ by $\sum_{m\in\bN:\,m|n}\mu(n/m)=0$ 
and the latter equality in \eqref{eq:factorization}.
For every $n\in\bN$ and every $\lambda_0\in(P_n^*(\cdot,0))^{-1}(0)$,
by the former equality in \eqref{eq:factorization}, 
we have $(c_0(\lambda_0)=)0\in\Fix_f^{**}(\lambda_0,n)$, which with 
$(f_\lambda^n)'(0)=(f_\lambda^n)'(c_0(\lambda))=0\neq 1$ 
implies even $0\in\Fix_f^*(\lambda_0,n)$.
Then by the latter equality in \eqref{eq:factorization}, 
$\lambda_0$ is a zero of $P_n^*(\cdot,0)$ of order $d-1$
since any zero of $F_n$ is in fact simple.
\end{proof}

Recall the definitions of the sequences
$(\sigma_0(n))$ and $(\sigma_1(n))$ in $\bN$ (in Notation \ref{th:sigma}).

\subsection{Proof of \eqref{eq:superattprimitive}}
For every $n\in\bN$, 
the estimate \eqref{eq:superattconcrete} 
together with \eqref{eq:dynatomicdegree} and \eqref{eq:factorization} 
yields the following $L^1(\omega)$ estimate
\begin{gather}
\int_{\bP^1}\Bigl|\log|P_n^*(\cdot,0)|
-(d-1)\nu(n)\frac{g_{I_{c_0}}}{d}\Bigr|\omega
\le t_n^*+(d-1)C_0\cdot\sigma_0(n),\label{eq:superattprimitiverough}
\end{gather}
where we set 
\begin{multline*}
 t_n^*:=(d-1)\sum_{m\in\bN:\,m|n}t_m=(2\log d)\sigma_1(n)\\
+\Bigl((d+1)\log 2-2\log d+\frac{4C_{B_f}}{d-1}
+(d-1)\log\bigl(\sqrt{2}+1\bigr)\Bigr)\sigma_0(n).
\end{multline*}
Recall that $H_1$ is by definition the component of $H_f$ containing $0$, and set
\begin{align*}
C_0^*
:=&\pi+\int_0^\infty\frac{2r}{(1+r^2)^2}
\log^+\frac{r}{\sup\{t>0:\bD(t)\subset H_1\}}\rd r\\
=&C_0-\int_{H_1}G_{H_1}(\cdot,0)\omega.
\end{align*}
In the rest of this subsection, 
for every $n>1$, we also point out a slightly better estimate
\begin{gather}
\int_{\bP^1}\Bigl|\log|P_n^*(\cdot,0)|
-(d-1)\nu(n)\frac{g_{I_{c_0}}}{d}\Bigr|\omega
\le t_n^*+(d-1)C_0^*
\label{eq:superattprimitivebetter}
\end{gather}
than \eqref{eq:superattprimitiverough}. 
In particular, by Green's theorem, for every $\phi\in C^2(\bP^1)$ and every $n>1$,
we have
\begin{gather}
 \left|\int_{\bP^1}\phi\rd\left(\Per_f^*(n,0)-\nu(n)\cdot T_f\right)\right|
 \le\biggl(\sup_{\bP^1}\biggl|\frac{\rd\rd^c\phi}{\omega}\biggr|\biggr)
 \cdot \bigl(t_n^*+(d-1)C_0^*\bigr),\tag{\ref{eq:superattprimitive}$'$}
\label{eq:superattprimitiveconcrete}
\end{gather} 
which implies \eqref{eq:superattprimitive}.

\begin{proof}[Proof of \eqref{eq:superattprimitivebetter}]
 For every $n\in\bN$, by \eqref{eq:factorization} and \eqref{eq:lowerconcrete}, we have
 \begin{gather}
 \sup_{B_f}\bigl|\log|P_n^*(\cdot,0)|\bigr|\le t_n^*,\tag{\ref{eq:lowerconcrete}$'$}\label{eq:lowerprimitive}
 \end{gather}
 which is a counterpart to \eqref{eq:lowerconcrete}.
 Fix $n>1$. By \eqref{eq:lowerprimitive},
 $\inf_{\lambda\in B_f}|P_n^*(\lambda,0)|\ge e^{-t_n^*}$. 
 As in the proof of Claim \ref{th:small} in Section {}\ref{sec:Selberg},
 let $\mathcal{F}^*$ be the family of all 
 components of $(P_n^*(\cdot,0))^{-1}(\bD(e^{-t_n^*}))$. 
 By Lemma \ref{th:preparatory}
 and the description of $H_f$ in Subsection \ref{sec:DH},
 every $V\in\mathcal{F}^*$ is a piecewise real analytic
 Jordan domain in $H_f\setminus(I_{c_0}\cup H_1)$ now, and
 for every $V\in\mathcal{F}^*$, 
 the restriction $P_n^*(\cdot,0)|V:V\to\bD(t_n^*)$ is a proper holomorphic
 mapping of degree $d-1$ now and $\#(((P_n^*(\cdot,0))^{-1}(0))\cap V)=1$. 
 For every $V\in\mathcal{F}^*$, letting $\lambda_V$ be the unique point in
 $((P_n^*(\cdot,0))^{-1}(0))\cap V$, by Myrberg's theorem \cite{Myrberg33},
 we now have  
\begin{gather*}
 \log\frac{e^{-t_n^*}}{|P_n^*(\cdot,0)|}
 =G_{\bD(e^{-t_n^*})}(P_n^*(\cdot,0),0)
 =(d-1)\cdot G_V(\cdot,\lambda_V) \quad\text{on }V.
\end{gather*} 
Recalling $t_n^*\ge 0$,
 by a computation similar to that in the proof of Claim \ref{th:small} in
 Section \ref{sec:Selberg},
we have
 \begin{multline*}
 \int_{(P_n^*(\cdot,w))^{-1}(\bD(e^{-t_n^*}))}\Bigl|\log|P_n^*(\cdot,0)|
 -\nu(n)(d-1)\frac{g_{I_{c_0}}}{d}\Bigr|\omega\\
 \le\omega((P_n^*(\cdot,0))^{-1}(\bD(e^{-t_n^*})))t_n^*+(d-1)C_0^*.
 \end{multline*}
 Moreover, by the same argument as that in the proof of Claim \ref{th:harmonic}
 in Section \ref{sec:Selberg},
 we also have
$\sup_{\bC\setminus (P_n^*(\cdot,0))^{-1}(\bD(e^{-t_n^*}))}\bigl|\log|P_n^*(\cdot,0)|-\nu(n)(d-1)d^{-1}g_{I_{c_0}}\bigr|\le t_n^*$. 
Hence \eqref{eq:superattprimitivebetter} holds.
\end{proof}

\subsection{Proof of \eqref{eq:nonrepaveraged}}
As an application of \eqref{eq:superattprimitivebetter}, we also point out 
the following $L^1(\omega)$ estimate
\begin{gather}
\int_{\bP^1}\biggl|\int_0^{2\pi}
\log|P_n^*(\lambda,re^{i\theta})|\frac{\rd\theta}{2\pi}
-\nu(n)(d-1)\frac{g_{I_{c_0}}}{d}\biggr|\omega(\lambda)
\le t_n^*+2(d-1)C_0^*\tag{\ref{eq:superattprimitivebetter}$'$}
\label{eq:averaged}
\end{gather}
for every $n>1$ and every $r\in(0,1]$
(cf.\ \cite[2.\ in Theorem 3.1]{BassanelliBerteloot11}). 
In particular, by Green's theorem, for every $\phi\in C^2(\bP^1)$, every $n>1$, and every $r\in(0,1]$, we will have
\begin{multline}
  \left|\int_{\bP^1}\phi\rd\left(\int_0^{2\pi}\Per_f^*(n,re^{i\theta})\frac{\rd\theta}{2\pi}-\nu(n)\cdot T_f\right)\right|\\
 \le\biggl(\sup_{\bP^1}\biggl|\frac{\rd\rd^c\phi}{\omega}\biggr|\biggr)
 \cdot(t_n^*+2(d-1)C_0^*),\tag{\ref{eq:nonrepaveraged}$'$}
\end{multline}
which implies \eqref{eq:nonrepaveraged}.

\begin{proof}[Proof of \eqref{eq:averaged}]
 For every $n\in\bN$ and every $\lambda\in\bC\setminus(H_f\setminus I_{c_0})$,
 we have $\inf_{z\in\Fix_f^{**}(\lambda,n)}|(f_{\lambda}^n)'(z)|\ge 1$.
 Recall the description of components of $H_f\setminus I_{c_0}$ 
 in Subsection \ref{sec:DH}. For every $n\in\bN$, 
 letting $H_n^*$ be the union of all components $U$ of 
 $H_f\setminus I_{c_0}$ such that $n_U=n$ (so e.g.\ $H_1^*=H_1$),
 there is a holomorphic function $\lambda\mapsto z_\lambda$ 
 on $H_n^*$ such that for every $\lambda\in H_n^*$,
 $z_\lambda\in\Fix_f^{**}(\lambda,n)$ and that 
$(f_{\lambda}^n)'(z_\lambda)\equiv\phi_U(\lambda)$ 
on each component $U$ of $H_n^*$.
 Fix $n>1$ and $r\in(0,1]$, and set
$H_n^*(r):=\{\lambda\in H_n^*:\,(f_{\lambda}^n)'(z_\lambda)\in\bD(r)\}
=\bigcup_{U:\text{ a component of }H_n^*}\phi_U^{-1}(\bD(r))$.
For every $\lambda\in\bC$,
by the definitions of $P_{f,n}^*$ and $p_{f,n}^*$, 
we have
 \begin{multline*}
 \int_0^{2\pi}\log|P_n^*(\lambda,re^{i\theta})|\frac{\rd\theta}{2\pi}
 =\frac{1}{n}\sum_{z\in\Fix_f^{**}(\lambda,n)}\log\max\{r,|(f_{\lambda}^n)'(z)|\}
  -\nu(n)\log d\\
 =\log|P_n^*(\lambda,0)|
 +\begin{cases}  
\displaystyle   \frac{1}{n}\sum_{j=0}^{n-1}\log\frac{r}{|(f_{\lambda}^n)'(f_\lambda^j(z_\lambda))|} & \text{if }\lambda\in H_n^*(r),\\
   0 & \text{if }\lambda\in\bC\setminus H_n^*(r),
  \end{cases}
 \end{multline*}
 which with \eqref{eq:superattprimitivebetter} and the chain rule yields
 \begin{multline*}
\int_{\bP^1}\biggl|\int_0^{2\pi}
\log|P_n^*(\lambda,re^{i\theta})|\frac{\rd\theta}{2\pi}
-\nu(n)(d-1)\frac{g_{I_{c_0}}}{d}\biggr|\omega(\lambda)\\
 \le \bigl(t_n^*+(d-1)C_0^*\bigr)+\int_{H_n^*(r)}
 \log\frac{r}{|(f_{\lambda}^n)'(z_\lambda)|}\omega(\lambda).
 \end{multline*}
 For every component $V$ of $H_n^*(r)$, 
 letting $U$ be the component of $H_n^*(=H_n^*(1))$ containing $V$, 
 the restriction $\phi_U|V:V\to\bD(r)$
 is a proper holomorphic mapping of degree $d-1$,
 so letting $\lambda_V$ be the unique point in $V\cap\phi_U^{-1}(0)$, 
 by Myrberg's theorem \cite{Myrberg33}, 
 we have 
\begin{gather*}
 \log\frac{r}{|(f_{\lambda}^n)'(z_\lambda)|}
 =G_{\bD(r)}((\phi_U|V)(\lambda),0)=(d-1)\cdot G_V(\lambda,\lambda_V) 
\quad\text{on }V.
\end{gather*}
Noting that $H_n^*\subset H_f\setminus(I_{c_0}\cup H_1)$, 
 by a computation similar to that in the proof of Claim \ref{th:small} in
 Section \ref{sec:Selberg}, we have
\begin{gather*}
  \int_{H_n^*(r)}\log\frac{r}{|(f_{\lambda}^n)'(z_\lambda)|}\omega(\lambda)
 \le(d-1)\cdot C_0^*.
\end{gather*}
Hence \eqref{eq:averaged} holds. 
\end{proof}

\begin{acknowledgement}
 The author thanks Professors Thomas Gauthier and Gabriel Vigny for 
 many discussions, and the referee for a very careful scrutiny
 and invaluable comments.
 This work was completed during the author's visit 
 to Laboratoire Ami\'enois de Math\'ematique Fondamentale et Appliqu\'ee, 
 Universit\'e de Picardie Jules Verne in Spring 2016, and the author thanks for 
 the hospitality and the financial support there.
 This research was partially supported by JSPS Grant-in-Aid for Young Scientists (B), 24740087.
\end{acknowledgement}

%\bibliographystyle{jipsj}  
%\bibliography{papers,books} 

\def\cprime{$'$}

\end{document}